\documentclass[11pt,a4paper]{article}
\usepackage{amssymb} 
\usepackage{amsmath} 
\usepackage{amsthm, hyperref}

\newtheorem{thr}{Theorem}[section]
\newtheorem{lem}[thr]{Lemma}

\newtheorem{conj}[thr]{Conjecture}
\theoremstyle{definition}
\newtheorem{defi}[thr]{Definition}

\DeclareMathOperator{\h}{h}
\DeclareMathOperator{\cm}{cm}
\DeclareMathOperator{\ccl}{ccl}

\title{Hadwiger's conjecture implies a conjecture of F{\"u}redi-Gy{\'a}rf{\'a}s-Simonyi}

\author{Stijn Cambie\footnote{Department of Mathematics, Radboud University Nijmegen, Postbus 9010, 6500 GL Nijmegen, The Netherlands. Email: \href{mailto:S.Cambie@math.ru.nl}{S.Cambie@math.ru.nl}. This work has been supported by a Vidi Grant of the Netherlands Organization for Scientific Research (NWO), grant number $639.032.614$.} }%
\date{}
\begin{document}

\maketitle

\begin{abstract}
	One of the most important open problems in the field of graph colouring or even graph theory is the conjecture of Hadwiger. This conjecture was the inspiration for many mathematical works, one of them being the work of F{\"u}redi, Gy{\'a}rf{\'a}s and Simonyi in which they ``risked'' to conjecture the precise bound for a graph with independence number $2$ to contain a certain connected matching. 
	We prove that their conjecture would be a corollary of Hadwiger's conjecture or equivalently if their risky conjecture would be false, then Hadwiger's conjecture would be false as well.
\end{abstract}

\section{Introduction}\label{sec:intro}

	We use the following standard notation for graphs and certain graph parameters.
	Let $\delta(G), \omega(G), \alpha(G)$ and $\chi(G)$ be the minimum degree, the clique number, the independence number and the chromatic number of a graph $G$ respectively.
	Let $\eta(G)$ be the largest value $t$ such that $G$ has a $K_t$ minor. 
	This parameter is called the Hadwiger number or contraction clique number of the graph $G$, representing the largest $t$ such that $K_t$ can be derived with edge deletions and edge contractions starting from $G.$ 
	In some other works it is also denoted with $\ccl(G)$ or $\h(G).$
	
	We say the disjoint subgraphs from $G$, $H$ and $H'$, are connected if there are vertices $u \in H$ and $v \in H'$ such that $uv$ is an edge from $G.$ We can extend this to a collection of subgraphs.
	
\begin{defi}[a connected collection of subgraphs]
	A collection of disjoint subgraphs of $G$ is (pairwise) connected if every two subgraphs of the collection are connected. 
\end{defi}

	The Hadwiger number is exactly equal to the maximum number of components/ subgraphs in a (pairwise) connected collection.
	An example of a connected collection is a connected matching, which is a collection of disjoint edges which are pairwise connected.
	Let $\cm(G)$ (or just $\cm$ when the underlying graph is clear from the context) be the size of a largest connected matching of $G$.

Hadwiger's conjecture~\cite{hadw} from $1943$ states that a graph $G$ having no $K_{t+1}$ minor, can be coloured with $t$ colours. Nowadays it is still one of the deepest and most famous unsolved problems in graph theory. By investigation of the random graph $G_{n, \frac 12}$, it is known that the conjecture is true for almost all graphs by~\cite{BCE80}. A survey on the conjecture can be found in~\cite{S16}.

\begin{conj}[Hadwiger, \cite{hadw}]\label{Hadw}
	For every graph $G$, we have $\eta(G) \ge \chi(G).$
\end{conj}

Due to the basic relation that $\chi(G) \ge \frac{n}{\alpha(G)}$, a weaker conjecture would be that $\eta(G) \ge \frac{n}{\alpha(G)}.$ This weakening was first studied by Duchet and Meyniel~\cite{DM82}.
A particular interesting case is the one where $\alpha=2$. Due to the relation with induced $K_{1,2}$s, which are called seagulls, this case is called the Seagull problem. 

\begin{conj}[Seagull problem]\label{weakHadw}
	For every graph $G$ with $\alpha(G)=2$, we have $\eta(G) \ge \frac n2.$
\end{conj}

As proven in e.g.~\cite[Thr. 3.1]{PST03}, this weaker problem is equivalent with Hadwiger's conjecture for $\alpha=2.$ 
This conjecture is proven in certain cases, e.g. when $G$ does not contain an induced $C_4$ or an induced $C_5$ by \cite[Thr. 5.2\&5.5]{PST03}
But in general, it is still unsolved and it might be a possible direction to disprove Hadwiger's conjecture.

It was observed in \cite{KPT, FGS} that to prove that $\eta(G) \ge c n$ for some $c> \frac 13$, it is necessary and sufficient to find a linear connected matching.
In the paper by F{\"u}redi, Gy{\'a}rf{\'a}s and Simonyi~\cite{FGS}, the authors ``risk the stronger conjecture" about the following precise version on the question of determining $\cm(G)$.

\begin{conj}[\cite{FGS}]\label{conjFGS}
	Let $G$ be a graph with independence number $\alpha =2$ and order $n=4t-1.$ Then $\cm \ge t.$
\end{conj}

The conjecture would be sharp due to the graph $G=K_{2t-1} \cup K_{2t-1}.$
The current best known lowerbound for $\cm$ in this setting is of the order $\Theta( n^{4/5} \log^{1/5} n )$ by \cite[Thr.3]{Fox10}.

Our main result is the observation that Conjecture~\ref{conjFGS} would be a corollary of Hadwiger's conjecture~\ref{Hadw}. So for those believing Hadwiger's conjecture is true, it is not a risky conjecture. On the other hand people who think Seagull's problem, Conjecture~\ref{weakHadw}, is false and hence also Hadwiger's conjecture, may try to prove this by disproving the Conjecture from~\cite{FGS}.

\begin{thr}\label{hadw=>F}
	If Conjecture~\ref{Hadw} is true, then so is Conjecture~\ref{conjFGS}. Equivalently, if Conjecture~\ref{conjFGS} is false, then so is Hadwiger's conjecture.
\end{thr}

More precisely we will prove that if Conjecture~\ref{conjFGS} is not true for a certain value of $t$, the Seagull's problem is false for $n=4t-1.$

The main observation in the proof for this is that in graph with $n=4t-1$, $\alpha=2$ and $\cm \le t-1$, a connected collection of vertices and edges contains at most $\cm$ components.
This is proven in Lemma~\ref{k1+k2<m}. 

\section{Proof of Theorem~\ref{hadw=>F}}

We start with the important observation that a graph with independence number $\alpha =2$ has large clique number or large minimum degree. To see this, note that for every vertex $v$, the complement of $N[v]$ has to be a clique. This implies that $n-(\delta +1) \le \omega.$
Using this observation, we now prove the following two lemmas saying that if any connected matching is reasonable small, so is the clique number and the number of components in a connected collection of vertices and edges.

\begin{lem}\label{omega<m}
	If a graph $G$ of order $n=4t-1$ satisfies $\alpha=2$ and $\cm\le t-1$, then $\omega \le \cm.$
\end{lem}

\begin{proof}
We will assume $\omega >\cm$ in this proof and derive a contradiction.
If $G$ is disconnected, then it is the union of two cliques, one having order at least $2t$ which would imply $\cm \ge t$.
Hence $G$ is connected.
Note that $\cm \ge \lfloor \frac {\omega}2 \rfloor$, so $\omega \le 2t-1$ and
$ \delta \ge n-1 - \omega\ge  2t-1.$
Take a clique $A$ of order $\omega$ and let $B=G \backslash A.$
Let $M$ be the largest matching in the bipartite $G[A,B]$.
If the size of $M$ is at least $\omega$, we are done.
Note that $A \cup M$ contain at least $2t$ vertices, since otherwise there was a vertex in $A \backslash M$ having a neighbour in $B \backslash M$ since $\delta >  \lvert A \cup M \rvert -1.$
But this implies that we can extend the matching $M$ in $A \cup M$ to a matching of size at least $t$ since $A$ is a clique, which is a contradiction again with $ \cm \le t-1$. 
\end{proof}

\begin{lem}\label{k1+k2<m}
Let $G$ be a graph of order $n=4t-1$ with $\alpha=2$ and $\cm\le t-1$ if there are $k_1$ vertices and $k_2$ edges which are pairwise connected, then $k_1+k_2 \le \cm.$
\end{lem}

\begin{proof}
Assume the contrary and take a connected collection of $k_1$ vertices and $k_2$ edges with $k_1+k_2=\cm+1\le t$ and $k_2$ maximal among these. Note that $ \delta \ge n-1 - \omega\ge 3t-1$ (using $\omega \le t-1$ by Lemma~\ref{omega<m}), while the connected collection contains at most $k_1+2k_2\le 2\cm+1\le 2t-1$ vertices as $k_2 \le \cm$.
	But this implies that any vertex of the $k_1$ vertices has a neighbour different from the vertices belonging to the connected collection. So we can form an additional edge which is connected, i.e. $k_2$ was not taken maximal, contradiction. Hence $k_1+k_2\le \cm.$	
\end{proof}

Now we can prove our main result by applying Lemma~\ref{k1+k2<m}.

\begin{proof}[Proof of Theorem~\ref{hadw=>F}]
	Assume a graph $G$ of order $n=4t-1$ satisfies $\alpha=2$ and $\cm \le t-1$ (i.e. $G$ is a counterexample to Conjecture~\ref{conjFGS}), we will prove that $\eta(G)<2t,$ which would imply that Conjecture~\ref{weakHadw} is false and hence Hadwiger's conjecture is false as well.
	Assume the contrary, so there exists a $K_{2t}$-minor in $G$.
	Let $k_1$ be the number of vertices, $k_2$ be the number of edges and $k_3$ be the number of components of order at least $3$ which form the minor (connected collection).
	We know $k_1+k_2 \le t-1$ due to Lemma~\ref{k1+k2<m} and by definition that $k_1+k_2+k_3=2t$ and $k_1+2k_2+3k_3 \le 4t-1$.
	Hence we can derive a contradiction as follows 
	$$6t = 3(k_1+k_2+k_3) \le k_1+2k_2+3k_3 + 2(k_1+k_2) \le 4t-1 +2(t-1)=6t-3.$$
\end{proof}

\subsection*{Acknowledgement}

The author is grateful to Zoltan F{\"u}redi, Andras Gy{\'a}rf{\'a}s, Ross J. Kang and G{\'a}bor Simonyi for their remarks on earlier thoughts and presentation of the results.

\bibliographystyle{abbrv}
\bibliography{imp_a2}

\end{document}